\documentclass[preprint,11pt]{elsarticle}

\usepackage[utf8]{inputenc}
\usepackage[T1]{fontenc}
\usepackage[colorlinks=true]{hyperref}
\usepackage{amsfonts,amsmath,amssymb,amsthm,bbm,enumerate,enumitem,geometry,wrapfig,url,yfonts,fancyhdr}
\usepackage{multirow,bm}
\usepackage{subfigure,graphicx,color,mathrsfs,ulem,scalerel}
\everymath{\displaystyle}

\newtheorem{tw}{Theorem}[section]
\newtheorem{corollary}[tw]{Corollary}

\newtheorem{proposition}[tw]{Proposition}

\theoremstyle{definition}
\newtheorem{definition}[tw]{Definition}
\newtheorem{example}[tw]{Example}
\newtheorem{remark}[tw]{Remark}

\renewenvironment{proof}[1][\textbf{Proof}]{\noindent\textbf{#1.} }{\hfill\qed\vspace{2.5ex}}

\newcommand\mN{{\mathbb N}}

\newcommand\bF{\mathbf{F}} 
\newcommand\bFa{\mathbf{F}_{\ast}} 
\newcommand\bFam{\mathbf{F}_{\ast,M}} 
\newcommand\bFb{\mathbf{F}_{b}} 
\newcommand\bFs{\bF_0^2} 
\newcommand\bM{\mathbf{M}} 

\newcommand\bMf{\mathbf{M}_{\mathrm{f}}}

\newcommand\cE{{\mathcal E}}

\newcommand\sA{\mathsf{A}}

\newcommand{\rH}{\mathrm{H}}

\newcommand{\rS}{\mathrm{S}}

\newcommand\aA[2][E]{\sA(#2|#1)}
\newcommand\aAi[3][E]{\sA^{#3}(#2|#1)}
\newcommand\aAd[3][E]{\sA_{#3}(#2|#1)}

\newcommand\sbullet[1][.5]{\mathbin{\vcenter{\hbox{\scalebox{#1}{$\bullet$}}}}}

\newcommand\bA{\mathscr{A}}
\newcommand\bpA{\mathscr {A}_{\sbullet|}}
\newcommand\beA{\mathscr {A}_{|\sbullet}}
\newcommand\bpeA{\mathscr {A}_{\sbullet|\sbullet}}

\newcommand\bE{{\mathscr {E}}}

\newcommand\mub[1][\bA]{\mu_{#1}}

\newcommand{\bmu}{\boldsymbol{\mu}}
\newcommand{\bnu}{\boldsymbol{\nu}}

\newcommand\bmub[1][\bpeA]{\bmu_{#1}}
\newcommand\bnub[1][\bpeA]{\bnu_{#1}}

\renewcommand\ge{\geqslant}
\renewcommand\le{\leqslant}

\newcommand\md{\,{\mathrm{d}}}
\newcommand\intl{\int\limits}

\newcommand\agg{\mathsf{Ag}}

\newcommand{\mI}[1]{\mathbbm{1}_{#1}}

\newcommand{\noi}{\noindent}
\newcommand\st{\star}
\newcommand{\tx}{\textstyle{}}

\usepackage{xcolor}\definecolor{darkgreen}{rgb}{0,0.75,0}

\newcommand{\mb}{\color{blue}{}}
\newcommand{\mk}{\color{red}{}}
\newcommand{\mi}{\color{darkgreen}{}}

\begin{document}
\begin{frontmatter}

\title{Generalized level measure based on a~family \\of conditional aggregation operators}

\author{Micha{\l} Boczek\corref{cor1}\fnref{label1}}
\ead{michal.boczek.1@p.lodz.pl}
\cortext[cor1]{michal.boczek.1@p.lodz.pl}

\author{Ondrej Hutn\'ik\fnref{label2}}
\ead{ondrej.hutnik@upjs.sk}

\author{Marek Kaluszka\fnref{label1}}
 \ead{marek.kaluszka@p.lodz.pl}
\author{Miriam Kleinov\'{a}\fnref{label2}}
\ead{miriam.kleinova@student.upjs.sk}

\address[label1]{Institute of Mathematics, Lodz University of Technology, 90-924 Lodz, Poland}
\address[label2]{Institute of Mathematics, Pavol Jozef \v Saf\'arik University in Ko\v sice, 040-01 Ko\v sice, Slovakia}

\begin{abstract}
Extending the concept of level measure $\mu(\{f\ge a\})$ we introduce a generalized level measure based on a~family of conditional aggregations operators.
We investigate in detail several basic properties, including connections with the family of level measures, the generalized survival function and the transformation of monotone measures to hyperset. Applications in scientometrics and information science are described.
\end{abstract}


\end{frontmatter}

\section{Introduction}\label{sec:1}
This paper deals with a~study of the concept of the generalized level measure being a~generalization of the level measure  $\mu(\{x\in X\colon f(x)\ge a\})$ using in many mathematical formulas such as the Choquet integral \cite{choquet}, the Sugeno integral \cite{sug74}, the seminormed fuzzy integral \cite{BHH} and their generalizations \cite{BHH19,lucca2017,MesSt19} (see Sec.~\ref{sec:2.3} for details).
These operators have many applications, e.g., in decision-making process \cite{grabisch1996},  risk theory \cite{kal17}, or fuzzy reasoning methods \cite{lucca2017}.
The survival function, $\mu(\{x\in X\colon f(x)>a\}),$ resembles the level measure and using in the above mentioned operators it can change their value (see Example~\ref{ex:3.21} below).
In~\cite{BHHK21}, a~generalization of the survival function  based on the conditional aggregation operator has been introduced and some basic properties have been studied.
The main motivation behind its definition were real life problems.
In this paper we present the corresponding generalization of the level measure, which allows to solve other problems (see Sec.~\ref{sec:3.2}).

The paper is organized as follows.
Section~\ref{sec:2.1} provides basic notations and definitions we work with.
Next, we recall the concept of the conditional aggregation operator introduced in~\cite{BHHK21}. In Section~\ref{sec:2.3}, we present motivating examples justifying the concept of the generalized level measure -- the main topic of this work.
In Section~\ref{sec:genlev} we introduce formal definition of the generalized level measure with a~number of simple examples.
Then, we study several basic properties (Sec.~\ref{sec:3.1}), we present connections with the family of level measures  (Sec.~\ref{sec:3.2}), and the connection with the generalized survival function introduced in \cite[Sec.~4]{BHHK21} (Sec.~\ref{sec:3.3}). Applications of the introduced concept for producing new scientometric indices and transformation of monotone measure to hyperset are given in Section~\ref{sec:4}.

\section{Background and motivations}\label{sec:2}

\subsection{\textbf{Basic notations}}\label{sec:2.1}

In this paper, unless stated otherwise, $X$ is a~nonempty set and $\Sigma$ is a~$\sigma$-algebra of subsets of $X.$ Hereafter, $\Sigma^0=\Sigma\setminus \{\emptyset\}$ and  $[k]=\{1,\ldots,k\}$ for fixed $k\in\mN=\{1,2,\ldots\}.$

By $\bF$ we denote the set of all $\Sigma$-measurable nonnegative bounded functions on $X.$
For $f,g\in\bF,$ we write $f\le g$ if $f(x)\le g(x)$ for all $x\in X$ and $0_X(x)=0$ for all $x\in X.$
Moreover, $\mI{A}$ denotes the indicator function of a~set $A,$ that is, $\mI{A}(x)=1$ if $x\in A$ and $\mI{A}(x)=0$ otherwise, and $\mI{}(S)$ denotes the indicator function of a~logical sentence $S,$ that is, $\mI{}(S)=1$ if $S$ is true and $\mI{}(S)=0$ otherwise.
We adopt the standard conventions: $\tx{\inf_{x\in \emptyset} f(x)}=\infty$ and $\tx{\sup _{x\in \emptyset} f(x)=0}.$
For any $a,b\in [0,\infty],$ let $a\wedge b=\min\{a,b\}$ and $a\vee b=\max\{a,b\}.$

A~{\it monotone measure} on $\Sigma$  is a~nondecreasing set function $\mu\colon \Sigma\to [0,\infty],$ i.e., $\mu(B)\le\mu(C)$ whenever $B\subset C$ for $B,C\in\Sigma$ with $\mu(\emptyset)=0,$ and $\mu(X)>0,$
where ``$\subset$'' and ``$\subseteq$'' mean the proper inclusion and improper inclusion, respectively.
We denote the class of all monotone measures  on $\Sigma$  by $\bM.$ If $\mu\in \bM$ is such that $\mu(X)<\infty,$ then $\mu$ is \textit{finite}. For $\mu,\nu\in\bM$ we write $\mu\le \nu$ whenever $\mu(E)\le \nu(E)$ for any $E\in\Sigma.$ A~set function $\mu\colon 2^X \to [0,\infty]$  is the~\textit{counting measure} if $\mu(B)=|B|$ for any finite $B$ and $\mu(B)=\infty$ otherwise, where $|B|$ means the cardinality  of  $B.$

For $a\in [0,\infty)$ we define the \textit{$a$-level set} as $\{f\ge a\}=\{x\in X\colon f(x)\ge a\},$ where $f\in\bF.$ This concept is also known in the literature as the~weak upper level set \cite[Ch.~4]{denneberg00}, or $a$-upper level set \cite[p.~36]{durante-sempi}, or $a$-cut set  \cite[p.~15]{lee}.
The \textit{level measure} is  the function $[0,\infty)\ni a\mapsto \mu(\{f\ge a\})$ with $(f,\mu)\in\bF\times \bM$ \cite[Sec.\,2.2]{BHH4}.

\subsection{\textbf{Overview of conditional aggregation operators}}\label{sec:2.2}
A~crucial concept used in this paper is a~conditional aggregation operator introduced in~\cite{BHHK21} and applied to define new operators \cite{BHK22} covering many known nonadditive operators in the literature.

\begin{definition}\label{def:conagr}
Let $E\in \Sigma^0.$ A~map $\aA{\cdot}\colon \bF\to [0,\infty]$
is said to be a~\textit{conditional aggregation operator w.r.t. $E$} (CAO, for short) if it satisfies the following conditions:
    \begin{enumerate}[leftmargin=1.5cm,noitemsep]
        \item[$(C1)$]  $\aA{f}\le \aA{g}$ for any  $f,g\in\bF$ such that $f(x)\le g(x)$ for all $x\in E$;
        \item[$(C2)$]  $\aA{\mI{E^c}}=0,$ where $E^c=X\setminus E.$
    \end{enumerate}
The set $E$  will be called a~\textit{conditional set}.
\end{definition}

The value $\aA{f}$ can be interpreted as ``an  aggregated value of $f$ on  $E.$'' In other words, the CAO only depends on the value of the considered function defined on the conditional set. Some basic properties of CAOs are summarized in the following proposition.

\begin{proposition}(cf. \cite[Prop. 3.3]{BHHK21})\label{pro:2.2}
If $\aA{\cdot}$ is a~CAO w.r.t. $E\in \Sigma_0,$ then    $\aA{0_X}=0$ and $\aA{f}=\aA{f\mI{E}}$ for any  $f\in\bF.$
\end{proposition}

Several classes of CAOs used in this paper are introduced in the following definition.

\begin{definition}\label{def:properties}
Let $\theta\in (0,\infty).$ We say that a~CAO w.r.t. $E$ is:
\begin{itemize}[noitemsep]
    \item  \textit{(super)homogeneous of degree} $\theta$  if $\aA{\lambda f}(\ge)=\lambda^{\theta} \aA{f}$  for all  $\lambda\in (0,\infty)$ and all $f\in\bF$;

    \item \textit{idempotent} if $\aA{b\mI{X}}=b$ for all $b\in (0,\infty).$


     \end{itemize}
\end{definition}

A~\textit{family of conditional aggregation operators} (FCA, for short) will be denoted by a~calligraphic letter, i.e.,
\begin{align}\label{n:fca}
    \bA=\{\aA{\cdot}\colon E\in \cE\},
\end{align}
where $\mathcal{E}$ is a~\textit{paving}, i.e., a~subfamily of $\Sigma$ such that $\emptyset \in \mathcal{E}$ (see \cite[p.~15]{denneberg00}, \cite[p.~254]{hahn},  or \cite[p.~7]{Pap1995}). Put $\cE^0=\cE\setminus\{\emptyset\}.$
In order for the FCA to be well defined, from now on we consider that each element of the FCA satisfies $\aA[\emptyset]{\cdot}=\infty,$  unless stated otherwise\footnote{such situation appears only in Section~\ref{sec:3.3}, where we adopt the convention $\aA[\emptyset]{\cdot}=0$}.
Although the family $\bA$ depends on $\cE,$ we will not indicate it in our notation when there is no doubt of confusion.  For $X=\mN$ and $\cE=\{[n]\colon n\in\mN\}\cup\{\emptyset\},$ the FCA $\bA$ generalizes the concept of aggregation operator $\agg(\cdot)$ presented by Calvo et al. \cite[Def.~1]{Calvo}.

\begin{example}\label{ex:basicCAOs}
We will frequently use the following families of CAOs  with superscript notation:
\begin{itemize}[noitemsep]
    \item \label{ref:sum} $\bA^{\textrm{sum}}=\{\aAi{\cdot}{\textrm{sum}}\colon E\in \cE\}$ with $\sA^{\textrm{sum}}(f|E)=\tx{\sum_{x\in E} f(x)}$ for $E\in\cE^0$ and $X$ is countable;

    \item \label{ref:prod} $\bA^{\textrm{prod}}=\{\sA^{\textrm{prod}}(\cdot|E)\colon E\in \cE\}$ with $\sA^{\textrm{prod}}(f|E)=\tx{\prod_{x\in E} f(x)}$ for $E\in\cE^0$ and $X$ is countable;

    \item $\bA^{\inf}=\{\aAi{\cdot}{\inf}\colon E\in \Sigma\}$ with $\aAi{\cdot}{\inf}=\tx{\inf_{x\in E} f(x)}$ for $E\in\Sigma^0$;

    \item $\bA^{\sup}=\{\aAi{\cdot}{\sup}\colon E\in \Sigma\}$ with $\aAi{f}{\sup}=\tx{\sup_{x\in E} f(x)}$ for $E\in\Sigma^0.$
\end{itemize}
Let us underline that $\mathcal{E}=\Sigma$ in $\mathscr{A}^{\inf}$ and $\bA^{\sup}.$ Moreover, in the above  families of CAOs  it is only necessary to define CAOs w.r.t. conditional sets $E\in \cE^0,$ since for  the empty set we assume (unless stated otherwise) to be equal to $\infty,$ e.g., $\aAi[\emptyset]{f}{\sup}=\infty.$
\end{example}

For some other examples of CAOs, we refer to \cite{BHHK21,BHK22}.

\subsection{\textbf{Motivations}}\label{sec:2.3}

In~\cite{BHHK21} it is introduced a~generalization of the survival function. Starting from the formula
\begin{align*}
\mu(\{f>a\})&= \mu\big(X\setminus \{x\in X\colon f(x)\leqslant a\}\big)
\\&=\inf\big\{\mu (E^c)\colon \sup_{x\in E} f(x)\le a,\, E\in \Sigma\big\},\quad a\in [0,\infty), \, f\in \bF,
\end{align*}
where $\{f>a\}=\{x\in X\colon f(x)>a\}$,
and replacing ${\tx \sup_{x\in E} f(x)}$ and $\Sigma$ with a~CAO $\sA(f|E)$ and  $\cE,$ respectively,
the generalized survival function in the form
\begin{align}\label{n30}
    \mub^S(f,a)=\inf\big\{\mu(E^c)\colon \sA(f|E)\le a,\, E\in \cE\big\}
\end{align}
is obtained (see \cite[Def. 4.1]{BHHK21}). It turns out that this modification can be used in several applications related to decision-making processes (see \cite[Sec. 2 and 6]{BHHK21}).

In this paper we propose an extension of the level measure. Observe that the level measure can be rewritten as follows
\begin{align*}
    \mu(\{f\ge a\})&=\sup\{\mu(E)\colon E\subseteq \{f\ge a\},\,E\in \Sigma\}\notag
    \\&=\sup\{\mu(E)\colon \inf_{x\in E} f(x)\ge a,\, E\in \Sigma\},\quad  a\in [0,\infty), \, f\in \bF.
\end{align*}
Putting  $\sA(f|E)$ instead of ${\tx \inf_{x\in E} f(x)}$ and  $\cE$ instead of $\Sigma,$ respectively, we get the following generalization of the level measure
\begin{align}\label{n31}
    \mub(f,a)=\sup\big\{\mu(E)\colon \aA{f}\ge a,\, E\in \cE\big\}
\end{align}
(see Section~\ref{sec:genlev} for more details).

Formulas \eqref{n30} and \eqref{n31} can be used for evaluation of both the survival function and the level measure.
Indeed, observe that if $\bA=\{\sA(\cdot|E)\colon E\in \Sigma\}$ and $\sA^{\inf}(\cdot|E)\le \aA{\cdot}\le \sA^{\sup}(\cdot|E)$ for all $E\in \Sigma^0$, then for all $a\ge 0$ and any $(f,\mu)\in \bF\times \bM$ we get the following bounds
\begin{align}\label{n4}
  \mub^S(f,a) \le \mu(\{f>a\})\le \mu(\{f\ge a\})\le \mub(f,a).
\end{align}
There exist situations in which the generalized survival function \eqref{n30} cannot be used, but the generalized level measure \eqref{n31} can help to fix them.  As a~demonstration we consider the following motivation from scientometry.

\begin{example}\label{ex:2.5}
Let $X=\mN$, $\Sigma=2^{\mathbb{N}}$ and $\mu$ be the counting measure.
A~scholar with some publications is described by a~function $f\colon \mathbb{N}\to\mathbb{N}\cup\{0\}$ called a~\textit{scientific record} such that $f(1)\ge f(2)\ge\ldots.$
The positive value of $f(i)$ gives the number of citations of scholar's $i$th publication, whereas the value $f(i)=0$ means either a~paper with zero citations or a~nonexisting paper.
Then $\mu(\{f\ge k\})$ with $k\in\mN$ returns the number of scholar's papers having at least $k$ citations each. The modified formula \eqref{n31} with $\aA{f}={\tx \inf_{x\in E} f(x)},$ i.e.,
\begin{align}\label{n5}
    \mu(f,k)=\sup\big\{\mu(E)\colon \inf_{x\in E} f(x)\ge k,\, E\in \cE\big\}
\end{align}
admits various interpretations. For instance, if $\cE=\{\emptyset,E\},$ where $E$ consists of scholar's papers published just in the last five years, then $\mu(f,k)$ defined in \eqref{n5} gives a~number of papers having at least $k$ citations among the papers from this period. This approach may provide a~fairer comparison of researchers with different professional internship (e.g., a~young scientist beginning of his/her carrier versus a~well-enabled scientist with many years in science).
Obviously one can replace $\textstyle{\inf_{x\in E} f(x)}$ in \eqref{n5} with a~CAO $\aA{f}$ to change the criterion. More examples are provided in Section~\ref{sec:4}.
\end{example}

\begin{remark}
Formula 
\eqref{n31} with $\bA=\{\sA^{\inf}(\cdot|E)\colon E\in \cE\}$ uncovers a~connection between the level measure and inner set function. Recall that the inner set function of $\mu\in \bM$ (see \cite[p.~21]{denneberg00}) is defined as follows
$\mu_\st(B)=\sup\{\mu(E)\colon E\subseteq B,\, E\in \cE\}$
for any $B\in 2^X.$ If $f\in\bF,$ then for $a\ge 0$ we can write
\begin{align*}
    \mu_\st (\{f\ge a\})&=\sup\{\mu(E)\colon E\subseteq \{f\ge a\},\,E\in \cE\}\\
    &=\sup\big\{\mu(E)\colon \inf_{x\in E} f(x)\ge a,\, E\in \cE\big\}=\mub (f,a).
\end{align*}
\end{remark}

The level measure and the survival function appear interchangeably in commonly known integrals of nonnegative functions, such as the Choquet  and the Sugeno integral. However, for some extensions of the above operators, interchanging the level measure and the survival function can change their value.
The following example illustrates such situation.

\begin{example}\label{ex:3.21}
Let $\rS$ be a~semicopula\footnote{A~binary function $\rS\colon [0,1]^2\to [0,1]$ is called a~\textit{semicopula} if it is nondecreasing and fulfils $\rS(1, a)=\rS(a, 1)=a$ for any $a\in [0,1].$} such that $\rS(x,y)=0$ if $(x,y)\in [0, 0.5)\times [0,0.5]$ and $\rS(x,y)=x\wedge y$ otherwise. Consider $f=0.5\mI{B}$ and monotone measure $\mu$ with $\mu(C)=0.5$ for any $C\neq \emptyset$ and $\mu(\emptyset)=0.$ Then the seminormed fuzzy integral and the modified seminormed fuzzy integral with the survival function of $f$ w.r.t. $\mu$ take the form:
\begin{align*}
    \sup_{a\in [0,1]} \rS\big(a,\mu(\{ f\ge a\})\big)&=\sup_{a\in [0,0.5]} \rS(a,0.5)\vee \sup_{a\in (0.5,1]} \rS(a,0)=0.5,\\
    \sup_{a\in [0,1]} \rS\big(a,\mu(\{ f> a\})\big)&=\sup_{a\in [0,0.5)} \rS(a,0.5) \vee \sup_{a\in [0.5,1]} \rS(a,0)=0.
\end{align*}
\end{example}
\bigskip

\section{Generalized level measure\label{sec:genlev}}

Inspired by stochastic processes, we consider families of concepts from Section~\ref{sec:2} such as monotone measures and  families of CAOs depending on the additional parameter $t.$
Note that a~family $\bmu=(\mu_t)_{t\in[0,1]}$ with
$\mu_t\in \{\nu\in\bM\colon \nu(X)=1\}$ is known as the~level dependent capacity on $(X, \Sigma)$, see~\cite[Def.\,3.1]{MesiarSmrek2015}.
Its special case is the~Markov kernel~\cite{Kallenberg}, where each $\mu_t$ is a~probability measure on
$(X, \Sigma)$ and for each $E\in \Sigma$, the function $[0,1]\ni t\mapsto \mu_t(E)$ is $\Sigma$-measurable.

By $\bMf$ we denote the family of monotone measures $\bmu=(\mu_t)_{t\in [0,\infty)}$ such that $\mu_t\in \bM$ for any $t.$
We write $\bmu\le\bnu$ for $\bmu,\bnu\in\bMf$ if $\mu_t(E)\le \nu_t(E)$ for all $E\in\Sigma$ and all $t.$
We say that $\bmu\in \bMf$ is:
\begin{itemize}[noitemsep]
    \item \textit{constant} if $\mu_t=\mu$ for any $t,$ where $\mu\in \bM$;
    \item \textit{nondecreasing} (resp. \textit{nonincreasing}) if $\mu_{s}\le \mu_{t}$ (resp. $\mu_{s}\ge \mu_{t}$) for any $s,t$ such that $s<t.$
\end{itemize}

From now on, we make use of $\bE=(\cE_t)_{t\in [0,\infty)}$ as a~\textit{process of pavings}, i.e., $\emptyset\in \cE_t$ with $\cE_t\subseteq \Sigma$ for any $t.$
Hereafter, $\cE_t^0=\cE_t\setminus \{\emptyset\}.$
A~process of pavings  $\bE$ is said to be:
\begin{itemize}[noitemsep]
   \item \textit{constant} if $\cE_t=\cE$ for each $t,$ where $\cE$ is a~paving;
   \item  \textit{nondecreasing} (resp. \textit{nonincreasing}) if $\cE_s\subseteq \cE_t$  (resp. $\cE_t\subseteq \cE_s$) for any $s<t$;

   \item \textit{closed under finite unions} if $E_1\cup E_2\in \cE_t$ for any $E_1,E_2\in\cE_t$ and any $t.$
\end{itemize}

For a~process  of pavings $\bE$ we define a~\textit{parametric family of conditional aggregation operators} (pFCA, for short) as follows
\begin{align}\label{f:n1}
     \bpeA=\{\aAd{\cdot}{t}\colon E\in \cE_t, \, t\ge 0\},
\end{align}
where $\sA_t(\cdot|E)\colon \bF\to [0,\infty]$ is a CAO w.r.t. $E\in \cE_t^0$ for any $t.$
We suppose that $\aAd[\emptyset]{\cdot}{t}=\infty$ for any $t$ unless stated otherwise. The formula \eqref{f:n1} is well defined as $\cE_t\subseteq \Sigma.$
To distinguish among several cases of  parametric families of CAOs used in this paper and covered by \eqref{f:n1}, we adopt the following notation (the first and the second bullet in the subscript of $\bA$ indicate dependence of the family $(\sA_t)$ and $(\cE_t)$ on parameter $t,$ respectively):

\begin{itemize}[noitemsep]
    \item  $\bpA=\{\aAd{\cdot}{t}\colon E\in \cE, \, t\ge 0\},$ i.e., $\bE$ is constant;

    \item $\beA=\{\aA{\cdot}\colon E\in \cE_t, \, t\ge 0\},$ i.e., $\aAd{\cdot}{t}=\aA{\cdot}$ for all $t$;\label{p:pfca}

    \item $\bA=\{\aA{\cdot}\colon E\in \cE\},$ i.e., FCA (see \eqref{n:fca}).
\end{itemize}
We now introduce the most important notion of this paper.

\begin{definition}\label{def:gfs}
The \textit{generalized level measure} of  $f\in\bF$ w.r.t.~a~pFCA $\bpeA$ and $\bmu\in \bMf$ is defined as follows
\begin{align}\label{f:n2}
    \bmub(f,t)&=\sup\{\mu_t(E)\colon \aAd{f}{t}\ge t,\, E\in \cE_t\},\quad t\in [0,\infty).
\end{align}
\end{definition}

\medskip

Clearly,  for any $t$, the set $\{E\in \cE_t\colon \aAd{\cdot}{t}\ge t\}$ is nonempty  as $\aAd[\emptyset]{\cdot}{t}=\infty$ and $\emptyset\in\cE_t.$
Putting a~FCA $\bA$ (instead of $\bpeA$) in \eqref{f:n2}, we get
\begin{align}\label{f:n2a}
    \bmub[\bA](f,t)&=\sup\{\mu_t(E)\colon \aA{f}\ge t,\, E\in \cE\}, \quad t\ge 0.
\end{align}
Setting a~constant $\bmu\in\bMf$ in \eqref{f:n2a}, we obtain 
\begin{align}\label{f:n2b}
    \mub(f,a)=\sup\{\mu(E)\colon \aA{f}\ge a,\, E\in \cE\}, \quad a\ge 0.
\end{align}
We now give four simple examples of the generalized level measure~\eqref{f:n2b}.

\begin{example}\label{ex:4.3}
Let $c\in (0,\infty)$ and $\cE$ be a~paving.
    \begin{enumerate}[noitemsep]
     \item[(a)]
      Let $\mu(E)=c\mI{}(E=X)$ for all $E\in \Sigma$ and $X\in \cE.$
      Then $\mu_\bA(f,a)= c\mI{}(\aA[X]{f}\ge a)$ for any $a\ge 0$ and any $f\in\bF.$

    \item[(b)] If $\mu(E)=c\mI{}(E\neq \emptyset)$ for all $E\in \Sigma,$  we have $\mu_\bA(f,a)= c\mI{}(\aA[E]{f}\ge a\text{ for some }E\in \cE^0)$, where  $a\ge 0$ and $f\in\bF.$

        \item[(c)]
    Let $X=[3]$, $\Sigma=2^{X}$ and $f=0.25\mI{\{1\}}+0.75 \mI{\{2\}}+\mI{\{3\}}$. Put $\cE=\{\emptyset,\{1\},\{2\},\{2,3\}\},$     $\bA=\{\sA^{\sup}(\cdot|E)\colon E\in \cE\}$ and $\mu(\{1\})=1$, $\mu(\{2\})=\mu(\{2,3\})=0{.}5.$  Then
    the generalized level measure \eqref{f:n2b} takes the form
  \[
    \mub(f,a)=
    \begin{cases}
    1 & \text{if } a\in[0, 0{.}25], \\
    0{.}5      & \text{if }  a\in(0{.}25,1],\\
    0     & \text{if } a>1.
    \end{cases}
    \]
    Observe that if $a\in[0,0{.}25]$, then $\tx{\mub(f,a)=\sup_{E\in \cE} \mu(E)},$ but for $a>0.25$ the equality does not hold.

    \item[(d)] If FCA $\bA$ is such that $\cE \ni E\mapsto \sA(\cdot|E)$ is nondecreasing, and $X\in\cE$ or $\cE$ is a~finite chain\footnote{$\cE=\{\{E_l,\ldots,E_1\}\colon E_l\subseteq E_2\subseteq\ldots \subseteq E_1,\, E_i\in \Sigma,\, i\in [l]\}$ for some $l\in\mN$}, then \eqref{f:n2b} takes the form
    $\mub(f,a)=\tx{\sup_{E\in \cE} \mu(E)}$ for any $a<\tx{\sup_{E\in \cE^0} \aA{f}}$  and any $(f,\mu)\in\bF\times \bM.$
\end{enumerate}
\end{example}

\subsection{\textbf{Basic properties}}\label{sec:3.1}

In this section we present some basic properties of the generalized level measure~\eqref{f:n2}.  Clearly, these properties depend on  parametric families of CAOs. Individual properties of a~CAO will be extended to a~pFCA as follows: a~pFCA $\bpeA=\{\aAd{\cdot}{t}\colon E\in \cE_t, \, t\ge 0\}$ is said to \textit{possess a~property} $\mathcal{P}$ if for any $t$ and any $E\in \cE_t^0$ the operator $\aAd{\cdot}{t}\in \bpeA$ has the property $\mathcal{P}$.
Pro\-per\-ties $\mathcal{P}$, we will mainly deal with, are the (super)homogeneity of degree $\theta$ and the idempotency,
see Definition~\ref{def:properties}.
We will also work with monotonicity of parametric families of CAOs in the following sense. We say that $\bpeA$ is:
\begin{itemize}[noitemsep]
    \item  \textit{nondecreasing} (resp. \textit{nonincreasing}) if $\bE$ is nondecreasing (resp. nonincreasing) and $\sA_{s}(f|E)\le \sA_{t}(f|E)$ for any $E\in \cE_s^0$
    (resp. $\sA_{s}(f|E)\ge \sA_{t}(f|E)$ for any  $E\in \cE_{t}^0$ and  all $0\le s<t;$

    \item \textit{nondecreasing} (resp. \textit{nonincreasing}) \textit{w.r.t. sets} if for each $f\in\bF$ and each $t,$ the inequality $\aAd[C]{f}{t}\le\aAd[D]{f}{t}$  (resp. $\aAd[C]{f}{t}\ge\aAd[D]{f}{t}$) holds for any $C,D\in \cE_t^0$ with $C\subset D$.
\end{itemize}

It is clear that the map $a\mapsto\mub(f,a)$  is nonincreasing for any $(f,\mu)\in\bF\times\bM.$ But it is no more true for the generalized level measure based on the pFCA in general.

\begin{example}
Let $\mu,$ $f$  and $\bA=\{\sA(\cdot|E)\colon E\in \cE\}$ be such that $\mu_\bA(f,b)<\mu_\bA(f,a)$ with $0<a<b.$ Let $0<c<a$ and put $\bpA=\{\sA_t(\cdot|E)\colon E\in \cE,\,t\ge 0\},$ where $\sA_c(\cdot|E)=c\sA(\cdot|E)/b$ and $\sA_a(\cdot|E)=\sA(\cdot|E).$ Suppose that $\bmu=(\mu_t)_{t\ge 0}$ is any family of monotone measures such that $\mu_c=\mu$ and $\mu\le \mu _a.$
The map  $t\mapsto \bmu_{\bpA}(f,t)$ is not nonincreasing as
\begin{align*}
\bmu_{\bpA}(f,c)&=\sup\{\mu(E)\colon \sA(f|E)\ge b,\, E\in \cE\}=\mu_{\bA}(f,b)\\
&<\mu_{\bA}(f,a)\le \bmu_{\bpA}(f,a).
\end{align*}
\end{example}

Next proposition specifies conditions under which the map $t\mapsto \bmub(f,t)$ is nonincreasing.

\begin{proposition}
For any fixed $f\in \bF,$ the function $t\mapsto \bmub(f,t)$  is nonincreasing whenever $\bpeA$ is nonincreasing pFCA and $\bmu\in\bMf$ is nonincreasing.
\end{proposition}
\begin{proof}
 Let $0\le s<t.$ Since $\mu_s\ge \mu_t$ and $\sA_s(\cdot|E)\ge \sA_t(\cdot|E)$ for each $E\in \cE_t\subseteq \cE_s,$ so
\begin{align*}
    \bmub(f,s)&=\sup\big\{\mu_s(E)\colon \sA_s(f|E)\ge s,\,E\in \cE_s\big\}\\&\ge\sup\big\{\mu_t(E)\colon \sA_t(f|E)\ge t,\,E\in \cE_t\big\}= \bmub(f,t),
\end{align*}
which gives the desired result.
\end{proof}

Note that the function $t\mapsto \bmub(\cdot,t)$ need not be nondecreasing even if $\bmu$ and $\bpeA$ are nondecreasing.


Monotonicity w.r.t. functions, monotone measures and parametric families of CAOs is the result of the following proposition. Its proof is straightforward.

\begin{proposition}\label{pro:4.9}
Let $f,g\in\bF$ and $\bmu,\bnu\in\bMf.$ Then:
    \begin{itemize}[noitemsep]
    \item[(a)] $\bmub(f,t)\le \bmub(g,t)$ for  all $t$ whenever $f\le g$;

    \item[(b)] $\bmub(f,t)\le \bnub(f,t)$ for all $t$ whenever $\bmu\le\bnu$;

    \item[(c)] $\bmub(f,t)\le \bmub[\widehat{\bpeA}](f,t)$ for all $t$ whenever $\cE_t\subseteq\widehat{\cE}_t$ and $\aAd{f}{t}\le \widehat{\sA}_t(f|E)$ for all $E\in\cE_t$ and any $t,$  where  $\widehat{\bpeA}=\{\widehat{\sA}_t(\cdot|E) \colon E\in \widehat{\cE}_t,\,t\ge 0\}$ is a~pFCA.
      \end{itemize}
\end{proposition}

Using Proposition~\ref{pro:2.2} we can prove the following result.

\begin{proposition}\label{prop:3.6}
For each $(f,\bmu)\in\bF\times \bMf$ and each $t\ge 0$ the following statements are true:
    \begin{enumerate}[noitemsep]
        \item[(a)] $\bmub(0_X,t)=0$  for $t>0$ and $\bmub(0_X,0)=\sup\{\mu_0(E)\colon E\in \cE_0\}$;

         \item[(b)] $\bmub(f,0)=\sup\{\mu_0(E)\colon E\in \cE_0\}$;

        \item[(c)]  $\bmub (f,t)=\mu_t(X)$ whenever $X\in\cE_t$  and $t\le \aAd[X]{f}{t}.$
    \end{enumerate}
\end{proposition}

\begin{proposition}
Let $(f,\bmu)\in\bF\times\bMf.$ Then:
\begin{itemize}[noitemsep]
    \item[(a)] for fixed $t,$ $\bmub(f,t)=0$  if and only if $\mu_t(E)=0$ for every $E\in \cE_t^0$ such that
    $\aAd{f}{t}\ge t$;

    \item[(b)] if $\sA_t(\lambda\mI{X}|E)\le\lambda$ for any    $E\in\cE^0_t$ and any $t,\lambda \ge 0,$  then there exists $b>0$ such that $\bmub(f,t)=0$ for each $t>b.$
\end{itemize}
\end{proposition}
\begin{proof} The proof of (a) follows from definition. In (b), by~\cite[Prop. 3.8\,(a)]{BHHK21}, we have $\aAd{f}{t}\le\sA^{\sup}(f|E)\le  \sA^{\sup}(f|X)<\infty$ for each $E\in\cE^0_t$ and each $t,$ so there exists $b > 0$ such that  $\aAd{f}{t}< b$ for each $E\in \cE_t^0$ and each $t>b.$ Hence, $\bmub(f,t)=0$ for each $t>b.$
\end{proof}

We now examine values of the generalized level measure for the constant function. Hereafter, $\infty\cdot 0=0.$

\begin{proposition}
If  $\bpeA$ is idempotent, then
$\bmub(\lambda\mI{X},t)=\sup\{\mu_t(E)\colon E\in \cE_t\}\mI{[0,\lambda]}(t)$ for any $\lambda>0$ and any $t\ge 0.$ In particular, $\bmub(\lambda\mI{X},t)=\mu_t(X)\mI{[0,\lambda]}(t)$ if $X\in\cE_t.$
\end{proposition}

The proof is immediate from definitions. Removing idempotency assumption of pFCA we get the following result.

\begin{proposition}\label{pro:4.13}
Let $D\in \Sigma^0.$
Assume that $\bpeA$ is such that $\aAd{\mI{E}}{t}=1$ for all $E\in\cE_t^0$ and all $t>0.$
Suppose also that
$\aAd{\mI{D}}{t}=0$ if $E\cap D^c\neq \emptyset,$
$E\in\cE_t^0$ and $t>0.$
Then  $\bmub(\mI{D},t)=\sup\{\mu_t(E)\colon E\subseteq D,\, E\in \cE_t\}$
for $t\in (0,1]$ and $\bmub(\mI{D},t)=0$ for $t>1.$
\end{proposition}
\begin{proof}
By the definition, $\bmub(\mI{D},t)=\sup\{\mu_t(E)\colon \aAd{\mI{D}}{t}\ge t,\,E\in\cE_t\}$ for $t\ge 0.$
If $E\subseteq D,$   $E\in\cE_t^0$ and $t>0,$ then by Proposition~\ref{pro:2.2} we get $\aAd{\mI{D}}{t}=\aAd{\mI{D\cap E}}{t}=\aAd{\mI{E}}{t}=1.$
Otherwise, $\aAd{\mI{D}}{t}=0.$
Hence, for any $t\in (0,1],$ we have
 \begin{align*}
     \bmub(\mI{D},t)&=\sup\{\mu_t(E)\colon \sA_t(\mI{D}|E)\ge t,\,E\subseteq D,\,E\in\cE_t\}
     \\&   =\sup\{\mu_t(E)\colon E\subseteq D,\,E\in\cE_t\}.
\end{align*}
Clearly, $\bmub(\mI{D},t)=0$ for $t>1.$
\end{proof}

The following  CAOs satisfy the assumptions of Proposition~\ref{pro:4.13}: $\sA_t(f|E)=\tx{h_E(\sA^{\inf}(f|E))}$ and $\sA_t(f|E)=\sA\big(h_E(\sA^{\inf}(f|E))f|E\big)$ for any $E\in\cE^0_t$ and any $t,$ where $h_E$ is a~nondecreasing and nonnegative function such that $h_E(0)=0$ and $h_E(1)=1$, and $\sA$ is a~CAO such that $\sA(\mI{E}|E)=1$ for any $E\in \Sigma^0.$

\begin{proposition}\label{pro:4.14159}
Let $c>0,$  $D\in \cE_t^0$ for any $t\in [0,c],$  $\tx{\sup_{E\in\cE_t} \mu_t(E)}=\mu_t(D)$  and $\aAd[D]{\mI{D}}{t}\ge t$  for any $t\in [0,c].$
Then $\bmub(\mI{D},t)=\mu_t(D)$ for any $t\in [0, c].$
\end{proposition}
\begin{proof}
Let $t\in [0,c].$ Clearly,  $D\in \{E\in \cE_t\colon \sA_t(\mI{D}|E)\ge t\}\subseteq \cE_t.$ As $\tx{\sup_{E\in\cE_t} \mu_t(E)=\mu_t(D)},$ we have
$\tx{\sup \{\mu_t(E)\colon \sA_t(\mI{D}|E)\ge t,\,E\in \cE_t\} =\mu_t(D)}.$
\end{proof}




Let $\bFs\subseteq \bF^2$  be a~nonempty set.
We say that a~pFCA $\bpeA$ is  \textit{$c$-quasi-superadditive}
(resp. \textit{$c$-quasi-subadditive}) on $\bFs$ if $c\in [0.5,1]$ and $\sA_t(f+g|E)\ge c(\sA_t(f|E)+\sA_t(g|E))$  (resp. $c\in[1,\infty)$ and $\sA_t(f+g|E)\le c(\sA_t(f|E)+\sA_t(g|E))$  for any $(f,g)\in\bFs,$  any $E\in \cE_t^0$ and any $t>0.$
Note that each pFCA is $0.5$-quasi-superadditive on $\bF^2.$
Moreover, the family $\bpeA=\bA^{\sup}$ is $c$-quasi-superadditive on $\bFs=\{(f,g)\in \bF^2
\colon f,g\text{ are comonotone}\}$
for any $c\in [0.5,1],$ as $\sA^{\sup}(f+g|E)=\sA^{\sup}(f|E)+\sA^{\sup}(g|E)$ for any $E\in \Sigma$ and any  $(f,g)\in \bF^2_0.$

\begin{tw}\label{tw:3.17}
Let a~process of pavings $\bE$ be closed under finite unions. Assume that $\bpeA$ is a~pFCA nondecreasing w.r.t. sets and $c$-quasi-superadditive on $\bFs$.
 \begin{itemize}[noitemsep]
     \item[(a)] If $\bmu$ is nonincreasing and $\bpeA$ is nonincreasing, then for all $t>0$ and all $(f,g)\in\bFs$
    \begin{align*}
         \bmub(f+g,c t)\ge \bmub(f,t)\vee\bmub(g,t).
    \end{align*}

     \item[(b)]  If $\bmu$ is nondecreasing and $\bpeA$ is nondecreasing, then for all $t>0$ and all $(f,g)\in \bFs$
    \begin{align*}
         \bmub(f+g,2c t)\ge \bmub(f,t)\vee\bmub(g,t).
    \end{align*}

      \item[(c)] If  
      $c=0{.}5$, then for all  $t>0,$ all $\bmu\in\bMf$ and all $(f,g)\in\bFs$
    \begin{align*}
        \bmub(f+g,t)\ge \bmub(f,t)\vee\bmub(g,t).
    \end{align*}
\end{itemize}
\end{tw}

\begin{proof}
Ad (a) From  $c$-quasi-superadditivity of pFCA $\bpeA$ on $\bFs$, we obtain
$$\{E\in\cE_{2t} \colon \aAd{f+g}{2t} \ge 2c t\}\supseteq \{E\in\cE_{2t} \colon \aAd{f}{2t}+\aAd{g}{2t} \ge 2t\},
$$
where $c\in [0{.}5, 1].$
By monotonicity of $\bE$ and $t\mapsto \sA_t(\cdot|\cdot),$ we have
\begin{align*}
    \{E\in\cE_{2ct} \colon \aAd{f+g}{2ct} \ge 2c t\}\supseteq \{E\in\cE_{2t} \colon \aAd{f}{2t}+\aAd{g}{2t} \ge 2t\}.
\end{align*}
Set $\Pi=\{(B,D)\in \cE_{2t}^2  \colon \aAd[B\cup D]{f}{2t}+\aAd[B\cup D]{g}{2t}\ge 2t\}.$
The family $\bE$ is closed under finite unions, so
\begin{align}\label{2n10}
    \bmub(f+g,2ct)&\ge \sup\big\{\mu_{2ct}(E)\colon \aAd{f}{2t}+\aAd{g}{2t}\ge 2t,\, E\in \cE_{2t}\big\}\notag
    \\&=\sup\big\{\mu_{2ct}(B\cup D)\colon (B,D)\in \Pi\big\}.
\end{align}
Put
$\Pi^f=\{(B,\emptyset)\in \cE_{2t}^2
\colon \aAd[B]{f}{2t}\ge t\}$ and  $\Pi^g=\{(\emptyset,D)\in \cE_{2t}^2
\colon \aAd[D]{g}{2t}\ge t\}.$
Since $\sA_{2t}(f|\emptyset)=\infty$ and $\bpeA$ is nondecreasing w.r.t. sets, we have
\begin{align*}
 \Pi^f\cup \Pi^g&\subseteq
\{(B,D)\in \cE_{2t}^2\colon \aAd[B]{f}{2t}\ge t,\, \aAd[D]{g}{2t}\ge t\}
\\&\subseteq
\{(B,D)\in \cE^2_{2t}\colon \aAd[B\cup D]{f}{2t}\ge t,\, \aAd[B\cup D]{g}{2t}\ge t\}\subseteq \Pi.
\end{align*}
Hence, by \eqref{2n10} we get
\begin{align*}
    \bmub(f+g,2ct)&\ge \sup\big\{\mu_{2ct}(B)
    \colon (B,\emptyset)\in \Pi^f\} \vee \sup\{\mu_{2ct}(D) \colon (\emptyset ,D)\in \Pi^g\big\}\notag
    \\&\ge \sup\big\{\mu_{2ct}(B)\colon \sA_{2t}(f|B)\ge t,\,B\in \cE_{2t}\big\} \vee \sup\big\{\mu_{2ct}(D)\colon \sA_{2t}(g|D)\ge t,\,D\in \cE_{2t}\big\}\notag
    \\&\ge \sup\big\{\mu_{2ct}(B)\colon \sA_{2t}(f|B)\ge 2t,\,B\in \cE_{2t}\big\} \vee \sup\big\{\mu_{2ct}(D)\colon \sA_{2t}(g|D)\ge 2t,\,D\in \cE_{2t}\big\}.
\end{align*}
Since $\bmu$ is nonincreasing,
we get
\begin{align*}
    \bmub(f+g,2ct)&\ge \sup\big\{ \mu_{2t}(B)\colon \aAd[B]{f}{2t}\ge 2t,\, B\in\cE_{2t}\}\vee \sup\big\{ \mu_{2t}(D)\colon   \aAd[D]{g}{2t}\ge 2t,\,D\in\cE_{2t}\big\},
\end{align*}
which completes the proof of part (a).

Ad (b)
By $c$-quasi-superadditivity  of pFCA $\bpeA$ on $\bFs$, we obtain
\begin{align*}
    \{E\in\cE_{2ct} \colon \aAd{f+g}{2ct} \ge 2c t\}\supseteq \{E\in\cE_{2ct} \colon \aAd{f}{2ct}+\aAd{g}{2ct} \ge 2t\}.
\end{align*}
Set $\Pi=\{(B,D)\in \cE_{2ct}^2 \colon \aAd[B\cup D]{f}{2ct}+\aAd[B\cup D]{g}{2ct}\ge 2t\}.$ Clearly,
\begin{align}\label{n10}
    \bmub(f+g,2ct)&\ge \sup\big\{\mu_{2ct}(E)\colon \aAd{f}{2ct}+\aAd{g}{2ct}\ge 2t,\, E\in \cE_{2ct}\big\}\notag
    \\&=\sup\big\{\mu_{2ct}(B\cup D)\colon (B,D)\in \Pi\big\}.
\end{align}
Put
$\Pi^f=\{(B,\emptyset)\in\cE_{2ct}^2\colon \aAd[B]{f}{2ct}\ge t\},$ and  $\Pi^g=\{(\emptyset,D)\in \cE_{2ct}^2 \colon \aAd[D]{g}{2ct}\ge t\}.$ From the assumption that $\bpeA$ is nondecreasing w.r.t. sets, it follows that
$$
\Pi^f\cup \Pi^g\subseteq \{(B,D)\in \cE^2_{2ct}\colon \aAd[B\cup D]{f}{2ct}\ge t,\, \aAd[B\cup D]{g}{2ct}\ge t\}\subseteq \Pi.
$$
Hence, by \eqref{n10}, we get
\begin{align}
    \bmub(f+g,2ct)&\ge
    \sup\big\{\mu_{2ct}(B)\colon (B,\emptyset)\in \Pi^f\big\} \vee \sup\big\{\mu_{2ct}(D)\colon (\emptyset,D)\in \Pi^g\big\}\notag
    \\
    & = \sup\big\{\mu_{2ct}(B)\colon \sA_{2ct}(f|B)\ge  t,\,B\in \cE_{2ct}\big\} \vee \sup\big\{\mu_{2ct}(D)\colon \sA_{2ct}(g|D)\ge t,\,D\in \cE_{2ct}\big\}.\label{n:19}
\end{align}
As $\bmu$ and $\bpeA$ are nondecreasing and $c\in [0.5,1]$, we have
\begin{align*}
\bmub(f+g,2ct)&\ge \sup\big\{ \mu_{t}(B)\colon \aAd[B]{f}{2ct}\ge t,\, B\in\cE_{t}\big\}\vee\sup\big\{ \mu_t(D)\colon   \aAd[D]{g}{2ct}\ge t,\,D\in\cE_{t}\big\}
\\&\ge \sup\big\{ \mu_t(B)\colon \aAd[B]{f}{t}\ge t,\, B\in\cE_{t}\big\}\vee\sup\big\{ \mu_t(D)\colon   \aAd[D]{g}{t}\ge t,\,D\in\cE_{t}\big\}.
\end{align*}
The proof of part  (b) is complete.

Ad (c)  
The formula \eqref{n:19} with $c=0.5$ implies the statement.
\end{proof}




\begin{example}
Let $(\psi_t)_{t\ge 0}$ be a~family of nondecreasing functions such that $\psi_t\colon [0,\infty]\to[0,\infty],$ $\psi_t(0)=0$,  $\psi_t(x+y)\ge c(\psi_t(x)+\psi_t(y))$   and  $\psi_s(x)\le \psi_t(x)$ for each $t\ge 0$, each $0\le s\le t$ and each $x,y,$ e.g., $\psi _t(x)=\varphi(t)x^p$ with any $p\ge 1,$ $c\le 1,$ and any nondecreasing and nonnegative function $\varphi$ on $[0,\infty)$.
Then $\bpeA$ with $\sA_t(f|E)=\psi_t(\aAi{f}{\textrm{sum}})$ is a~pFCA which is $c$-quasi-superadditive on $\bF\times \bF$, nondecreasing and nondecreasing w.r.t. sets.
\end{example}

\begin{remark}
Under the assumptions of Theorem~\ref{tw:3.17}\,(b) we cannot determine, which in\-equa\-li\-ties from point (b) or (c) is better. This is due to the fact that the function $t\mapsto \bmub(f,t)$ is neither nondecreasing nor nonincreasing.
\end{remark}

\begin{proposition}\label{pro:3.19}
If $\bmu\in\bMf$ is nonincreasing, and  $\bpeA$ is a~$c$-quasi-subadditive on $\bF_0^2$ and nonincreasing pFCA, then
\begin{align}\label{eq:3.19}
\bmub(f+g,ct)\le \bmub(f,\lambda t)\vee \bmub(g,(1-\lambda)t)
\end{align}
for any $t\ge 0,$  any $\lambda\in (0,1),$  and  any $(f,g)\in\bFs.$
\end{proposition}
\begin{proof} Since $c\ge 1$ and pFCA $\bpeA$ is nonincreasing, we have $\cE_{ct}\subseteq \cE_t\subseteq \cE_{\lambda t}$ and
\begin{align*}
   \big\{E\in \cE_{ct}&\colon \aAd{f+g}{ct} \ge ct\big\}\subseteq  \big\{E\in \cE_{ct}\colon \aAd{f}{ct}+\aAd{f}{ct} \ge t\big\}
   \\ &\subseteq\big\{E\in \cE_{ct}\colon \aAd{f}{ct} \ge \lambda t\}\cup\{E\in \cE_{ct}\colon \aAd{g}{ct} \ge (1-\lambda)t\big\}
    \\ &\subseteq\big\{E\in \cE_{\lambda t}\colon \aAd{f}{\lambda t} \ge \lambda t\big\}\cup \big\{E\in \cE_{(1-\lambda) t}\colon \aAd{g}{(1-\lambda) t} \ge (1-\lambda)t\big\}.
\end{align*}
Hence, $\bmub(f+g,ct)\le \bmub(f,\lambda t)\vee \bmub(g,(1-\lambda) t),$ as claimed.
\end{proof}

Observe that the inequality \eqref{eq:3.19} is attainable for some $t$. Indeed, let $X$ be a~countable set, $\mu_t$ be the counting measure and $\bE$ be a~process of pavings such that $X\in \cE_t$ for all $t$.
Put $\sA_t(\cdot|E)=\aAi{\cdot}{\sup}$ for all $t.$ Take  $f=d\mI{X}$ and $g=2d\mI{X}$ with some $d>0.$
Then  $\bmub(f,(1-\lambda)t)=\bmub(g,\lambda t)=|X|$ and $\bmub(f+g,t)=|X|$ for all $t\in [0,d],$ so there is an equality in \eqref{eq:3.19} for all $t\in  [0,d]$ and all $\lambda\in (0,1).$


\begin{corollary}\label{cor:3.19}
If $\bmu\in\bMf$ is nonincreasing and $\bA$ is a~$1$-quasi-subadditive FCA on $\bF_0^2$, then
\begin{align*}
\bmub[\bA](f+g,t)\le \bmub[\bA](f,0.5 t)\vee \bmub[\bA](g,0.5 t)
\end{align*}
for any $t\ge 0$ and any $(f,g)\in\bF_0^2.$
\end{corollary}

\subsection{\textbf{Connections with the family of level measures}}\label{sec:3.2}

The formula \eqref{f:n2a} with $\bA^{\inf}$ takes the form
\begin{align}\label{b3a}
    \bmu_{\bA^{\inf}}(f,t)=\sup\big\{\mu_t(E)\colon \sA^{\inf}(f|E) \ge t,\, E\in \Sigma\big\}=\mu_t(\{f\ge t\})
\end{align}
for any $t\ge 0$ and any $(f, \bmu)\in\bF\times\bMf,$ which is a~particular case of level dependent set function introduced by Greco et al.~\cite{Greco2011}.
The family of $(\mu_t(\{f\ge t\}))_{t\ge 0}$ is called a~\textit{family of level measures}.

\begin{remark}\label{rem:collection}
In \eqref{b3a} we have to consider $\bA^{\inf}$ instead of $\beA^{\inf} = \{\sA^{\inf}(\cdot|E)\colon E\in \cE_t,\,t\ge 0\}.$
Otherwise, putting $\cE_t\subset \Sigma$ for any $t\ge 0$ in the left side in \eqref{b3a} instead of  $\Sigma$, the second equality need not hold.
Indeed, consider $\Sigma=\{\emptyset,B,B^c,X\},$ $\cE_t=\cE=\{\emptyset,B\}$ for any $t\ge 0$   and $f=c\mI{X},$ where $\emptyset \neq B\subset X$ and $c>0.$ Then
\begin{align*}
    \sup\big\{\mu_t(E)\colon \sA^{\inf}(f|E) \ge t,\, E\in \cE\big\}=\begin{cases}
   \mu_t(B) & \text{if } t\in [0,c],\\
  0 & \text{if } t\in (c,\infty).
\end{cases}
\end{align*}
In consequence, if $\mu_t(B)\neq \mu_t(X)$ for any $t\ge 0,$ we have
$$
\sup\{\mu_t(E)\colon \sA^{\inf}(f|E) \ge t,\, E\in \cE\}\neq \mu_t(\{f\ge t\})\; \;\textrm{ for any } t\in [0,c].
$$
More generally, if there exist  $B\in \Sigma$ and $t\ge 0$ such that $\mu_t(B)>{\tx \sup_{E\in \cE_t}\mu_t(E)},$ then for $f=t\mI{B}$ we have
$$\tx{\sup\{\mu_t(E)\colon \sA^{\inf}(f|E) \ge t,\, E\in \cE_t\}<\mu_t(\{f\ge t\})}.$$
\end{remark}

\medskip

Due to the observation in Remark~\ref{rem:collection}, from now on, we shall deal with the connection between $\mu_t(\{f\ge t\})$ and $ \bmu_{\bpA}(f,t)$ only for $\bpA=\{\sA_t(\cdot|E)\colon E\in \Sigma,\,t\ge 0\}$.
First, note that from Proposition~\ref{pro:2.2} we have $\sA_t(0_X|E)=0$ for any $E\in \Sigma^0$ and any $t,$ so using Proposition~\ref{prop:3.6}\,(a) we get
$\bmub[\bpA](0_X,t)=\mu_0(X)\mI{}(t=0).$ Therefore, $\mu_t(\{0_X\ge t\})=\bmub[\bpA](0_X,t)$ for any $t\ge 0$ and any $\bmu\in\bMf.$ Similarly, applying Proposition~\ref{prop:3.6}\,(b), we obtain $\bmub[\bpA](f,0)=\mu_0(\{f\ge 0\})$ for any $(f,\bmu)\in \bF\times \bMf.$ To exclude these trivial cases, we put $\bFa = \bF\setminus \{0_X\}.$

\begin{proposition}\label{pro:3.4}
Let $\bpA=\{\sA_t(\cdot|E)\colon E\in \Sigma,\,t\ge 0\}.$ If $\aAi{\cdot}{\inf}\le  \aAd{\cdot}{t}$  for any $t>0$ and any $E\in\Sigma^0,$ then $\mu_t(\{f\ge t\})\le \bmu_{\bpA}(f,t)$ for any $t>0$ and any $(f,\bmu)\in\bFa\times \bMf.$
\end{proposition}

Note that the assumption of Proposition~\ref{pro:3.4} is fulfilled for any pFCA $\bpA=\{\sA_t(\cdot|E)\colon E\in \Sigma,\,t\ge 0\}$ which is  superhomogeneous of degree $1$ with $\aAd{\mI{X}}{t}\ge 1$ for $t>0,$ $E\in \Sigma^0.$ Then by~\cite[Prop. 3.8\,(b)]{BHHK21} we have $\sA^{\inf}(\cdot|E)\le \sA_t(\cdot|E)$ for 
$t>0,$ $E\in \Sigma^0.$
Further, we are interested in the following property
\begin{equation}\label{nb5a}
 \bmu_{\bpA}(f,t)=\mu_t(\{f\ge t\})\,\, \text{ for any }\: t>0 \text{ and any }  (f,\bmu)\in \bFa\times \bMf.
\end{equation}
The next result provides a~characterization for the nonparametric class of CAOs.

\begin{tw}\label{tw:3.18}
Let $\bA=\{\sA(\cdot|E)\colon E\in \Sigma\}$ be a~FCA.
\begin{enumerate}[noitemsep]
    \item[(a)] Assume that    $\Sigma=\{\emptyset,X\}.$ Then~\eqref{nb5a} with $\bpA=\bA$ holds  if and only if  $\sA(\cdot|X)$ is idempotent.

    \item[(b)] Assume that  $\Sigma\neq \{\emptyset,X\}.$
    Then \eqref{nb5a} with $\bpA=\bA$
    is true  if and only if  $\bA=\bA^{\inf}.$
\end{enumerate}
\end{tw}

\begin{proof}
\noindent (a)
Since $\Sigma=\{\emptyset,X\},$ we have $\bFa=\{b\mI{X}\colon b>0\}.$
Observe that
\[\bmu_{\bA}(b\mI{X},t)
=\begin{cases}
 0 &\text{if } t>\aA[X]{b\mI{X}},\\
 \mu_t(X) &\text{if } t\le \aA[X]{b\mI{X}}
\end{cases}\]
for any $t,b>0,$ as $\sA(b\mI{X}|\emptyset)=\infty.$
Clearly,  the condition \eqref{nb5a} is equivalent to the idempotency of $\sA(\cdot|X)$ as $\mu_t(X)>0$ for all $t$.

\noindent (b) The implication ,,$\Leftarrow$'' is obvious (see \eqref{b3a}).\\
,,$\Rightarrow$''
Let $B\in\Sigma^0,$ $b\in (0,\infty)$ and $f=b\mI{B}.$
From \eqref{f:n2a} with $\cE=\Sigma$ and \eqref{nb5a}
we have
\begin{align}\label{n6}
    \bmu_{\bA}(f,b)=\sup\{\mu_b(E)\colon \aA{b\mI{B}}\ge b,\,E\in \Sigma\}=\mu_b(B)\; \text{ for any } \bmu\in\bMf.
\end{align}
If $B=X$, put $\mu_t(E)=\mI{}(E=X)$ for all $t,E.$ If $B\neq X,$ take $\mu_t$ such that $\mu_t(B)\in (0,1)$, $\mu_t(E)=1$ for all $E$ such that $B\subset E$ and $\mu_t(E)=0$ elsewhere for each $t.$ From  \eqref{n6} and arbitrariness of $b$ and $B,$ we obtain
\begin{align}\label{n7}
  \sA(b\mI{B}|B)\ge b \quad \text{for any } b>0 \text{ and  any } B\in \Sigma^0.
\end{align}
Moreover, by Proposition~\ref{pro:2.2}, (C1) and \eqref{n7}, we get
\begin{align*}
    \aA{f}=\aA{f\mI{E}}\ge \aA{(\inf_{x\in E}f(x))\mI{E}} \ge \sA^{\inf}(f|E)
\end{align*}
for any $f\in\bFa$ and any $E\in\Sigma^0.$ We will prove  the reverse inequality  by a~contradiction.
Suppose that there exist  $f\in\bFa$ and $D\in\Sigma^0$ such that $\aA[D]{f}>\sA^{\inf}(f|D).$
Put $t_0=\aA[D]{f}$ and let $\bmu=(\mu_t)$ be any family of monotone measures such that  $\mu_{t_0}(E)=1$ if $D\subseteq E$ and $\mu_{t_0}(E)=0$ otherwise. Then
\begin{align*}
   \bmu_{\bA}(f,t_0)&= \sup\{\mu_{t_0}(E)\colon \aA{f}\ge t_0,\,E\in\Sigma\}=1,\\
    \mu_{t_0}(\{f\ge t_0\})&=\sup\{\mu_{t_0}(E)\colon \sA^{\inf}(f|E)\ge t_0,\,E\in\Sigma\}=0,
\end{align*}
 as $\sA^{\inf}(f|E)\le \sA^{\inf}(f|D)$ for $D\subseteq E$, contradicting \eqref{nb5a} with $\bpA=\bA$.
Hence, $\bA=\bA^{\inf},$ as expected.
\end{proof}

Note that the implications ,,$\Rightarrow$'' in points (a) and (b) of Theorem~\ref{tw:3.18}  do not hold for a~pFCA $\bpA$ when replacing the condition ``for any  $(f,\bmu)\in \bFa\times\bMf$'' in \eqref{nb5a} by ``for any  $(f,\bmu)\in \bFam\times \bMf$'', where $\bFam=\{f\in \bFa\colon {\tx \sup_{x\in X} f(x)\le M}\}$ for some $M>0.$

\begin{example}\label{ex:3.17}
Let $M>0.$ Consider the pFCA $\bpA = \{\sA_t(\cdot|E)\colon E\in \Sigma,\,t\ge 0\}$ such that $\sA_t=0.5\,\sA^{\inf}$ whenever $t> M$, and $\sA_t = \sA^{\inf}$, otherwise. It is clear that
\begin{align}\label{n19}
    \bmu_{\bpA}(f,t)=\mu_t(\{f\ge t\})
\end{align}
holds for any $t\in (0,M]$ and any $(f,\bmu)\in \bFam\times \bMf.$  Clearly, $\mu_t(\{f\ge t\})=0$ for any $t>M$ and any $(f,\bmu)\in \bFam\times \bMf.$
Suppose that there exists $(f,\bmu)\in \bFam\times \bMf$ such that $\bmu_{\bpA}(f,t) >0.$ Hence there exists $E\in\Sigma^0$  such that $\sA_t(f|E)\ge t$ and $\mu_t(E)>0.$ Then $M\mI{X}\ge f$ and $0.5M=0.5\,\sA^{\inf}(M\mI{X}|E)\ge 0.5\,\sA^{\inf}(f|E)=\sA_t(f|E)\ge t>M,$ a~contradiction.
To sum up, the equality \eqref{n19} is true for any $t\in (0,\infty)$ and any $(f,\bmu)\in \bFam\times \bMf,$
but $\bpA\neq \bA^{\inf}.$
\end{example}

\medskip

Theorem~\ref{tw:3.18} results that in order to get the equality \eqref{nb5a}
for families of CAOs $\bA$, only the family $\bA^{\inf}$ can be used.
This highlights the role of the level measure in the developed theory. However, for certain families of monotone measures one can show that the generalized level measure  still may end up in the level measure~\eqref{b3a}.
For instance, consider  $\bpA^{\operatorname*{ess\,inf}_{\bmu}}=\{\sA^{\operatorname*{ess\,inf}_{\mu_t}}(\cdot|E)\colon E\in \Sigma,\,\mu_t\in\bM,\,t\ge 0\}$ with
\begin{align*}
   \aAi{f}{\operatorname*{ess\,inf}_{\mu_t}}=\operatorname*{ess\,inf}\nolimits_{\mu_t} (f\mI{E})=\sup\{a\in  [0,\infty)\colon \{f\mI{E}<a\}\in\mathcal{N}_{\mu_t}\}
\end{align*}
being the essential infimum of $f\in\bF$ on $E$ w.r.t. $\mu_t.$
Recall that $\mathcal{N}_\mu$ is the set of all \textit{null sets} w.r.t. $\mu\in\bM$, i.e.,  $N\in\mathcal{N}_\mu$ if $N\in \Sigma$ and $\mu(E\cup N) = \mu(E)$ for each $E\in \Sigma$ (see \cite[Def.~2.107]{grabisch2016} and \cite[Thm. 2.108\,(vii)]{grabisch2016}).

\begin{proposition}
     Let $\bmu\in\bMf$ be such that $\mu_t$ is continuous from below for any $t>0.$
Then for any $t>0$  and any 
$f\in\bFa$ we have
$$
\bmub[\bpA^{\operatorname*{ess\,inf}_{\bmu}}](f,t)=\mu_t(\{f\ge t\}).
$$
\end{proposition}

\begin{proof}
Clearly, $\sA^{\inf}(\cdot|E)\le \sA^{\operatorname*{ess\,inf}_{\mu_t}}(\cdot|E)$ for any $t$ and any $E\in\Sigma,$ as ${\tx \inf_{x\in E}f(x)}=\sup\{a\in  [0,\infty)\colon \{f\mI{E}<a\}=\emptyset\}.$ Then by Proposition~\ref{pro:3.4} we get
\begin{align}\label{b1}
   \mu_t(\{f\ge t\})\le  \bmub[\bpA^{\operatorname*{ess\,inf}_{\bmu}}](f,t)
\end{align}
for any $t>0.$
Now we shall show the validity of the reverse inequality. For a~fixed $t>0$, let us consider a~set $E\in\Sigma$ such that $\sA^{\operatorname*{ess\,inf}_{\mu_t}}(f|E)\ge t$.
Put $N_n=E\cap\{f<t-1/n\}$ for $n\in\mathbb{N}$.
Obviously, $N_n\subseteq N_{n+1}$.
From the definition of essential infimum and \cite[Thm. 2.108\,(iii)]{grabisch2016}, we have $N_n\in\mathcal{N}_{\mu_t}$ for any $n\in \mathbb{N}$, so we put $N=\tx{\bigcup_{n\in\mN} N_n}=\tx{ E\cap \bigcup_{n\in\mN} \{f<t-1/n\}= E\cap \{f<t\}}.$
By continuity from below of $\bmu$ and \cite[Thm. 2.108\,(v)]{grabisch2016}, we have $N\in\mathcal{N}_{\mu_t}$.
Furthermore,
\begin{align*}
    E\setminus N=E\cap \{f\ge t\}=(E\cap \{f\ge t\})\setminus N\subseteq \{f\ge t\}\setminus N.
\end{align*}
By \cite[Thm.~2.108\,(vi)]{grabisch2016} and monotonicity of $\mu_t,$ we have
\begin{align}\label{n8}
    \mu_t(E)=\mu_t(E\setminus N)\le \mu_t(\{f\ge t\}\setminus N)=\mu_t(\{f\ge t\}).
\end{align}
Since the inequality \eqref{n8} is satisfied for any set $E\in\Sigma$ with $\sA^{\operatorname*{ess\,inf}_{\bmu_t}}(f|E)\ge t$, we obtain
\begin{align}\label{b2}
 \sup\{\mu_t(E)\colon\sA^{\operatorname*{ess\,inf}_{\bmu_t}}(f|E)\ge t, E\in\Sigma\}\le \mu_t(\{f\ge t\}).
\end{align}
Combining inequalities \eqref{b1} and \eqref{b2} and from arbitrariness of $t,$ we get the result.
\end{proof}

\medskip
We have provided some conditions under which the equality $\bmub[\bpA](f,t)=\mu_t(\{f\ge t\})$ is true for all $t,f$. However, in general, these two objects are different in the sense that
it is not true that for each $(f,\bmu)\in\bF\times \bMf$ and each~pFCA $\bpA$ there exist $g\in\bF$ and $m_t\in\bM$ such that
$$
\bmub[\bpA](f,t)= m_t(\{g\ge t\})
$$
for every $t>0$.
This is confirmed in the following example.

\begin{example} \label{ex:4.8}
Let $X=[3]$ and $\Sigma=2^X.$ Consider a constant $\bmu\in\bMf$ defined as follows
$$\mu(\emptyset)<\mu(\{1\})<\mu(\{2\})<\mu(\{3\})<\mu(\{1,2\})<\mu(\{1,3\})<\mu(\{2,3\})<\mu(X).$$
For the function $f=1/2\cdot\mI{\{1\}}+1/3\cdot\mI{\{2\}}+1/4\cdot\mI{\{3\}}$ and the FCA $\bA^{\textrm{prod}}$ with $\cE=\Sigma$ (see p.~\pageref{ref:prod}), the generalized level measure $\mub[\bA^{\textrm{prod}}](f,a)$ takes
all eight possible values 
when varying $a\in (0,\infty)$  as given in Table~\ref{table}.

\begin{table}[h]
 \renewcommand*{\arraystretch}{1.7}
 \begin{center}
 \begin{tabular}{|c|c|c|c|c|c|c|c|c|}
 \hline
 $a$ & $\big[0,\tfrac{1}{24}\big]$ & $\big(\tfrac{1}{24},\tfrac{1}{12}\big]$ & $\big(\tfrac{1}{12},\tfrac{1}{8}\big]$ & $\big(\tfrac{1}{8},\tfrac{1}{6}\big]$ & $\big(\tfrac{1}{6}, \tfrac{1}{4}\big]$ & $\big(\tfrac{1}{4}, \tfrac{1}{3}\big]$ & $\big(\tfrac{1}{3}, \tfrac{1}{2}\big]$ & $\big(\tfrac{1}{2},\infty\big)$ \\ \hline
 $\mub[\bA^{\textrm{prod}}](f,a)$ & $\mu(X)$  & $\mu(\{2,3\})$ & $\mu(\{1,3\})$ & $\mu(\{1,2\})$ & $\mu(\{3\})$ & $\mu(\{2\})$ & $\mu(\{1\})$ & $0$\\\hline
 \end{tabular}
 \caption{Values of $\mub[\bA^{\textrm{prod}}](f,a)$ for function $f$ from Example~\ref{ex:4.8}.}
 \label{table}
 \end{center}
 \end{table}

\noi In contrary, varying $a\in (0,\infty)$ in $m_a(\{g\ge a\})$
for any function $g$ on $X$ and any
$(m_a)_{a\ge 0}\in \bMf$ one can obtain at most four different values.
\medskip
\end{example}

\subsection{\textbf{Connection with the generalized survival function}}\label{sec:3.3}

In this section we describe the relationship between the generalized level measure \eqref{f:n2} and the generalized survival function introduced in \cite[formula (12)]{BHHK21}. We first recall the latter concept.

\begin{definition}
Let $\widehat{\bpeA}=\{\widehat{\sA}_t(\cdot|E)\colon E\in \cE_t,\,t\ge 0\}$  be a~pFCA such that $\widehat{\sA}_t(\cdot|\emptyset)=0$ for any $t.$ The~\textit{generalized survival function of $f\in\bF$} w.r.t. $\widehat{\bpeA}$ and $\bmu\in \bMf$ is defined as
\begin{align}\label{n:gsf}
    \bmu_{\widehat{\bpeA}}^{S}(f,t)=\inf\big\{\mu_t(E^c)\colon \widehat{\sA}_t(f|E)\le t,\, E\in \cE_t\big\}, \quad t\ge 0.
\end{align}
\end{definition}

It is worth mentioning that the generalized survival function \eqref{n:gsf} is well defined as $E^c\in \Sigma$ for any $E\in\mathcal{E}_t$ and $t$, and the set
$\{E\in \cE_t\colon \widehat{\sA}_t(f|E)\le t\}$ is nonempty for all $t,$ since $\widehat{\sA}_t(\cdot|\emptyset)=0$  and $\emptyset\in \cE_t.$
The generalized survival function extends the concept of survival function $\mu(\{f>t\}),$ where $\mu\in\bM.$
Indeed, setting a~constant $\bmu$ and $\widehat{\bpeA}=\bA^{\sup}$  with $\sA^{\sup}(\cdot|\emptyset)=0$
in \eqref{n:gsf} we get the survival function (for more details, see \cite{BHHK21}).

For the benefit of the reader we summarize the main differences of the generalized level measure and generalized survival function:
\begin{itemize}[noitemsep]
    \item both notions are based on the pFCA  $\bpeA=\{\sA_t(\cdot|E)\colon E\in \cE_t,\, t\ge 0\}$, but with different conventions on the empty set;

    \item in the generalized level measure we measure the sets $E\in\cE_t$ for which $\sA_t(f|E)\ge t$, whereas in the generalized survival function we measure complements of sets $E\in\cE_t$ for which $\sA_t(f|E)\le t$;

    \item finally, in the generalized level measure we take supremum of all such values of $\mu_t(E)$, whereas in the generalized survival function we take infimum of all $\mu_t(E^c)$.
\end{itemize}
These observations suggest to describe their ``complementary'' or ``dual'' behaviour.
Assume that
$b\in (0,\infty).$
Given a~family $\bmu=(\mu_t)_{t\ge 0}$  of monotone measures such that $\mu_t(X)=\mu_0(X)<\infty$ for all $t$,
we define the \textit{dual family $\widehat{\bmu}=(\widehat{\mu}_t)_{t\ge 0}$
to $\bmu$}  by
\begin{align}\label{MK1}
\widehat{\mu}_t(E)=\mu_0(X)-\mu_{(b-t)_+}(E^c),\quad t\ge 0,\,E\in\Sigma,
\end{align}
where $a_+=\max\{a,0\}.$
It is clear that $\widehat{\bmu}\in\bMf$ with $\widehat{\mu}_t(X)=\mu_0(X)$ for any $t.$ Let $\bpA=\{\sA_t(\cdot|E)\colon E\in \cE,\, t\ge 0\}$ be a~pFCA such that $\sA_t(b\mI{X}|E)=b$ for any $t$ and any $E\in \cE^0.$ Define the \textit{dual pFCA to} $\bpA$ by
$\widehat{\bpA}=\{\widehat{\sA}_t(\cdot|E)\colon E\in \cE,\,t\ge 0\}$ with
$\widehat{\sA}_t(f|E)= b-\sA_{(b-t)_+}\big((b\mI{X}-f)_+|E\big)$
and $\widehat{\sA}_t(\cdot|\emptyset)=0.$
Clearly,  $\widehat{\sA}_t(b\mI{X}|E)=\sA_t(b\mI{X}|E)=b$  for any $t.$
Put  $f\in \bFb=\{f\in\bF\colon \textstyle{\sup_{x\in X} f(x)}\le b\}.$
Then, for each $(f,\bmu)\in \bFb\times \bMf$ and $t\in [0,b]$ we get
\begin{align*}
\bmub[\bpA](f,t)&=\sup\{\widehat{\mu}_0(X)-\widehat{\mu}_{b-t}(E^c)\colon \aAd{f}{t}\ge t,\,E\in \cE\}
\\&=\widehat{\mu}_0(X)-\inf\{\widehat{\mu}_{b-t}(E^c)\colon \aAd{f}{t}\ge t,\,E\in \cE\}
\\&=\widehat{\mu}_0(X)-\inf\{\widehat{\mu}_{b-t}(E^c)\colon \widehat{\sA}_{b-t}(b\mI{X}-f|E)\le b- t,\,E\in \cE\}
\\&=\mu_0(X)-\widehat{\bmu}^{S}_{\widehat{\bpA}}(b\mI{X}-f,b-t).
\end{align*}
To sum up, we get the following relationship between the generalized level measure and the generalized survival function.

\begin{proposition} Let $b\in(0,\infty)$, $\bmu=(\mu_t)_{t\ge 0}$ with $\mu_t(X)=\mu_0(X)<\infty$ for all $t$, and   $\widehat{\bmu}=(\widehat{\mu}_t)_{t\ge 0}$ be the dual family to $\bmu$ given by \eqref{MK1}.
If  $\widehat{\bpA}$ is the dual pFCA to a~pFCA $\bpA$  and $\sA_t(b\mI{X}|E)=b$ for any $t\ge 0$ and any $E\in \cE^0$, then
\begin{align}\label{MK2}
\bmub[\bpA](f,t)=\mu_0(X)-\widehat{\bmu}^{S}_{\widehat{\bpA}}(b\mI{X}-f,b-t)
\end{align}
for all $t\in [0,b]$ and all $f\in \bFb.$
\end{proposition}

\begin{example}
Let $\mu _t(E)=\tx{\sup_{x\in E}\pi_t(x)}$ be the possibility measure, where $\pi_t$ is a~nonnegative function on $X$ and $\tx{\sup_{x\in X}\pi_t(x)=1}$ for all $t.$  Put $b=1$ and $\sA_t(f|E)=\mathbb{E}_\mathsf{P}(f^{t}|E),$ where $\mathbb{E}_{\mathsf{P}}(\cdot|E)$ denotes the conditional expected value w.r.t. to a~probability measure $\mathsf{P}$.
Then \eqref{MK2} holds true for all $t\in (0,1)$ with $\widehat{\sA}_t\big(f|E)=\mathbb{E}_{\mathsf{P}}(\mI{X}-(\mI{X}-f)^{1-t}|E)$ for $f\in\bF_1$ and $\widehat{\mu}_t(E)=\nu_{1-t}(E),$ where $\nu_t$ is the necessity measure defined by $\nu_t(E)=
{\tx \inf_{x\in E^c}\{\mI{X}-\pi_t(x)\}}$ for any $E\in \Sigma.$
\end{example}

\section{Applications}\label{sec:4}

This section provides applications of the generalized level measure in scientometrics (with a~framework for unifying many existing indices and introducing new ones) and in transformations of a~monotone measure to hyperset.

\subsection{\textbf{Scientometrics}}
Let $X=\mN,$ $\Sigma=2^\mN,$  $\bA$ be a~FCA with $\cE=\{[n]\colon n\in\mN\}\cup \{\emptyset\}$  and  $\mu$ be the counting measure\footnote{
Note that the inclusion of emptyset into collection is a~technical assumption for the generalized indices to be well defined.}.
Assume that $f\colon \mN\to \mN\cup\{0\}$ is a~\textit{scientific record} i.e., $f(1)\ge f(2)\ge \ldots$
(cf. Example~\ref{ex:2.5}).
Consider the following parametric families of CAOs
\begin{align*}
    \bA_{1}=\big\{\tx{\frac{g(\aA{f})}{|E|}}\colon E\in\cE\big\}\quad \text{and}\quad
    \bA_2=\{g(\aA{f})\colon E\in\cE\},
\end{align*}
where $g$ is a~nondecreasing function such that $g(0)=0,$  $g(\infty)=\infty$ and $\sA\in\bA.$
Hereafter $x/0=\infty$ for any $x\in [0,\infty].$
Then the generalized level measures w.r.t. above parametric families of CAOs take the following forms
     \begin{align}
        \mub[\bA_1](f,a)&=\sup\big\{|E|\colon g(\aA{f})\ge a|E|,\, E\in \cE\big\},\label{index:n1}\\
         \mub[\bA_2](f,a)&=\sup\big\{|E|\colon g(\aA{f})\ge a,\, E\in \cE\big\}\label{index:n2}
    \end{align}
for any $a\ge 0,$ respectively. These two generalized level measures provide new scientometrics indices.  Below we present several special cases known in the literature.
\begin{enumerate}[noitemsep]
\item   Let $\bA=\bA^{\inf}.$ Then \eqref{index:n1} can be rewritten as follows
\begin{align}\label{index:n3}       \mub[\bA_1^{\text{inf}}](f,a)&=\max\{k\in \mN\colon g(f(k))\ge ak\}
\end{align}
for any $a\ge 0,$ as $\aAi{f}{\inf}=f(n)$ for $E=[n].$ If $g(x)=x,$ then the generalized level measure \eqref{index:n3} coincides with $\mathrm{h}_a$-index 
\cite[Def.\,2.4]{eck}, i.e.,
$\mub[\bA_1^{\text{inf}}](f,a)=\max\{k\in \mN\colon f(k)\ge a k\}.$
Next, for $a=1$ we get the Hirsch index, and for $a\to 0$ we have the $p$-index~\cite[Def.\,2.5]{eck} returning a~number of publications with at least one citation.
Further, considering $a=1$ and an increasing function $g$ with $g(0)=0$,
the formula~\eqref{index:n3} is the generalized Kosmulski index (cf. \cite{boczek20,deineko}). Clearly for $g(x)=\sqrt{x}$ we obtain the $\mathrm{h}(2)$-index \cite{kosmulski}.

It is worth noticing that some modifications of the Hirsch index have applications not only in scientometrics, but also in evaluating a~person’s output of Internet media. For instance $\mathrm{H}_{1000}$ is defined as the highest number $\mathrm{H}$ of videos with at least $\mathrm{H}\cdot 10^3$ views (see \cite{martinez}). Replacing $10^3$ with $10^5$ we get the index proposed by Hovden \cite{hovden} capturing both productivity and impact in a~single metric.
The scientific record  can also be understood as the number of count views of $i$th video in channel.

\item If $g(x)=\sqrt{x}$ and $\bA=\bA^{\textrm{sum}},$ then we get
    \begin{align*}
        \mub[\bA_1^\text{sum}](f,a)=\max\Big\{k\in\mN\colon \sum_{x=1}^{k} f(x)\ge (ak)^2\Big\}.
    \end{align*}
For $a=1$ we obtain $\mathrm{g}$-index \cite[Def.\,3.1]{eck}. Some modification of $\mathrm{g}$-index is also applied in information sciences similarly to Hirsch index (see point 1).

    \item Let $\aA{f}=\tx{\prod_{x\in E} (f(x))^{1/|E|}}$ for $E\in\cE^0$ and $g(x)=x.$ Then we obtain the $\mathrm{t}$-index \cite[formula~(4)]{tol}
    \begin{align*}
        \mub[\bA_1](f,1)=\max\Big\{k\in \mN\colon \prod_{x=1}^k (f(x))^{1/k}\ge k\Big\}.
    \end{align*}

     \item Let $\aA{f}=\tx{(\sum_{x\in E} 1/f(x))^{-1}}$ for $E\in\cE^0$ and $g(x)=x$. Then by \eqref{index:n2}, we get the $\mathrm{f}$-index \cite[formula~(3)]{tol}
    \begin{align*}
        \mub[\bA_2](f,1)=\max\Big\{k\in\mN\colon \Bigl(\sum_{x=1}^k 1/f(x)\Bigr)^{-1}\ge 1\Big\}.
    \end{align*}

\end{enumerate}

\subsection{\textbf{M2M-transformation on hyperset}}\label{sec:3.4}

Let $\mu\in\bM,$ $\mu(X)=1$ and $\Sigma$ be a~countable set.
Recall that $\widehat{\mu}\colon 2^{\Sigma}\to [0,1]$ is a~\textit{capacity on $2^{\Sigma}$}  if (a) $\widehat{\mu}(\emptyset)=0$; (b) $\widehat{\mu}(2^X)=1$; (c) $\widehat{\mu}(\widehat{D}_1)\le \widehat{\mu}(\widehat{D}_2)$ for all  $\widehat{D}_1\subset \widehat{D}_2\subseteq \Sigma$ (cf. \cite{YM}). The problem of transformation of $\mu$ to a~capacity
$\widehat{\mu}$ on $2^\Sigma$ (the M2M-transformation, for short)
is examined in~\cite[Sec.\,IV-V]{YM}, where the following transformations are given:
\begin{itemize}[noitemsep]
    \item  $\widehat{\mu}_1(\widehat{B})=\sup \{\mu(B)\colon B\in \widehat{B}\},$ where $\hat{B}\subseteq \Sigma$;

    \item ${\tx \widehat{\mu}_2(\widehat{B})=\sum _{B_j\in \widehat{B}}\nu(B_j)/\sum _{j=1}^{2^n}\nu(B_j)},$ where $\widehat{B}=\{B_j\}$ is a~family of all nonempty subsets of $X=[n]$ and $\nu(B)=g(\mu(B),|B|)$ with  a~function $g\colon [0,1]\times [n]\to [0,\infty)$  such that $g(x,k)>0$ for $x\in [0,1]$, $k\in [n]$.
\end{itemize}
Some another possible transformations include:
\begin{itemize}[noitemsep]
    \item  ${\tx \widehat{\mu}_3(\widehat{B})=\mu\big(\bigcup\{B_j\colon B_j\in \widehat{B}\}\big)};$
    \item ${\tx \widehat{\mu}_4(\widehat{B})=\sum _{B_j,B_i\in \widehat{B}}\mu(B_i\cup B_j)/\sum _{B_i,B_j\in \Sigma}\mu(B_i\cup B_j)}$;
    \item ${\tx \widehat{\mu}_5(\widehat{B})=1-\mu\big(\bigcap\{B_j^c\colon B_j\in \widehat{B}\}\big)}.$
\end{itemize}
Observe that $\widehat{\mu}_3(\widehat{B})=\widehat{\mu}_1(\widehat{B})$ if $\widehat{B}$ is a~chain, and $\widehat{\mu}_3(\widehat{B})={\tx \sum_{B_j\in \widehat{B}}\mu(B_j)}$ if $\widehat{B}$ is a~partition of some subset of $X$ and $\mu$ is additive.

\begin{example}
Let $X=\{x_1,x_2,\ldots,x_N\}$ denote all citizens who want to vote for one of the political parties $P_1,\ldots,P_m.$
Denote by  $B_i$ the set of voters of party $P_i$ and by $n_i$ the number of seats in parliament
for the party $P_i$ determined by some method of distribution, for example by the D'Hont method.
Denote  by $n$ the total number of seats in the parliament. A~party $P_j$ enters parliament if $|B_j|/N\ge 5\%.$  Clearly, the set function
$\mu\colon \Sigma\to [0,1],$ $\Sigma=2^X,$ defined by $\mu(B_j)=n_j/n$ for $j\in [m]$ is a~monotone measure, but the most often nonadditive.

Suppose that the parties $P_{1},\ldots,P_{k}$ have crossed the five percent threshold. Before the vote on a~bill, a~coalition $C=\{P_j\colon j\in J \subseteq\{1,\dots,k\}\}$
may be formed, where  $C$ is a~family of those parties in parliament whose leaders want to vote for the project. Any coalition $C$ determines
the family $\widehat{B}=\{B_j\colon j\in J\}$ of subsets of $X.$ Obviously, the  power of
any coalition  $C$ is given by the value of the transformation of $\mu$ given by
$\widehat{\mu}(\widehat{B})={\tx \sum _{B_j\in \widehat{B}}\mu(B_j)}$  if each envoy of any party votes as the leader of that party.
\end{example}

We propose to apply the generalized level measure to convert a~monotone measure $\mu$ from $\Sigma$ to $2^\Sigma$ as follows
\begin{align}\label{A11}
    \widehat{\mu}_6(\widehat{B})=\sup\big\{\mu(B)\colon \sA(f|B)\ge a,\, B\in \widehat{B}\cup\{\emptyset\}\}, \quad \widehat{B}\subseteq \Sigma^0
\end{align}
for a~fixed $\sA(\cdot|\cdot),$ $(f,\mu)\in \bF\times \bM$ and $a\in [0,\infty).$
Clearly, $\widehat{\mu}_6$ is a~monotone measure on $2^{\Sigma}$ provided that $\mu(X)=1$ and $\sA(f|X)\ge a.$ Moreover, $\widehat{\mu}_6\le \widehat{\mu}_1$
and $\widehat{\mu}_6=\widehat{\mu}_1$ for $a=0.$

\begin{example}
Consider a~system with components $b$ and $c,$ which can be selected from the sets $\{b_i\colon i\in [k]\}$ and $\{c_j\colon j\in [m]\},$ respectively.
Let $p(b_i)$  (resp. $p(c_j)$)  denote  the probability that the element $b_i$ (resp. $c_j$)  will not fail within a~given period.
If  $b$ and $c$  are connected in parallel and operate independently of each other, then the probability of survival of the system  is equal to $1-(1-p(b))(1-p(c))=p(b)+p(c)-p(b)p(c).$
Denote by $\pi(B)$ the price of the system $B=\{b,c\}$ and let $K$ be a~limited budget.
We say that $B$ is \textit{the best system} if it has the lowest
price among those whose probability of survival is not less than $p$, where $p$ is fixed.

Set $X=\{b_1,\ldots,b_k,c_1,\ldots,c_m\}$, $\widehat{B}=\{\{b_i,c_j\}\colon i\in [k],\,j\in [m]\},$
$\mu(B)=(1-\pi(B)/K)_+$ and $\sA(f|B)=p(b)+p(c)-p(b)p(c)$ for $B=\{b,c\},$ where $f(x)=p(x)$ for all $x\in X.$
Thus, choosing the best system comes down to calculating the value of $\widehat{\mu}_6(\widehat{B})$ given by \eqref{A11}.
An extension of this example to multicomponent systems is possible.
\end{example}

\section{Conclusion}

In this paper, we have proposed a~generalization of the level measure with its examples. We have presented some properties and applications related to scientometric indices. It turns out that the generalization of the survival function introduced in~\cite{BHHK21} cannot be used in this application, hence the need to introduce the~generalized level measure. When using Proposition~\ref{pro:3.4}, the level measure is, for some parametric families of CAOs, smaller than the generalized level measure.
From \eqref{n4} it follows that we may  give both  the lower and upper bound for the Choquet integral by means of new Choquet-like  functionals of the form
$$
\intl_0^\infty \mub^S(f,a)\md a\quad \text{and}\quad \intl_0^\infty \mub(f,a)\md a,
$$
and introduce novel classes of integrable functions $L_{\mu}^S$, cf. \cite{BHHK21}, and $L_{\mu},$ respectively. Clearly, $L_{\mu}\subseteq L_{\mu}^C\subseteq L_{\mu}^S$, where
$L^C_\mu$ is the set of Choquet integrable functions. A~detailed study and practical consequences are left for the future work.

\section*{Acknowledgement}
This work was supported by the Slovak Research and Development Agency under the contract
No. APVV-16-0337, and the grant VEGA 1/0657/22. The work is also cofinanced by internal grants vvgs-pf-2021-1782 and vvgs-pf-2022-2143. 

\end{document}